\documentclass[11pt,reqno]{amsart}
\usepackage[utf8x]{inputenc}
\usepackage{ucs}
\usepackage[english]{babel}
\usepackage{amsmath,slashed}
\usepackage{amsfonts}
\usepackage{mathtools}
\usepackage{amssymb}
\usepackage{graphicx}
\usepackage[left=2cm,right=2cm,top=2cm,bottom=2cm]{geometry}

\usepackage[final]{pdfpages} 
\usepackage{subcaption} 

\newlength{\intwidth}

\DeclareRobustCommand{\Bint}
   {\mathop{%
      \text{%
        \settowidth{\intwidth}{$\int$}%
        \makebox[0pt][l]{\makebox[\intwidth]{$-$}}%
        $\int$}}}

\usepackage{amssymb}
\usepackage[centertags]{amsmath}
\usepackage{verbatim}

\newcommand{\pd}{\partial}
\newcommand\om\omega
\newcommand\te\theta
\newcommand\Te\Theta
\newcommand{\al}{\alpha}
\newcommand{\si}{\sigma}
\newcommand{\Si}{\Sigma}
\def\RR{\mathbb{R}}
\def\ZZ{\mathbb{Z}}
\newcommand{\snabla}{\slashed{\nabla}}
\renewcommand\ell{{ l }}
\def\cN{\mathcal N}
\def\hu{\widehat u}
\def\cS{\mathcal S}

\newcommand\ep\epsilon
\newcommand\de\delta
\newcommand\La\Lambda
\newcommand\De\Delta
\newcommand\CC{\mathbb C}
\newcommand\Ga\Gamma
\newcommand\cE{\mathcal E}
\newcommand\R{_{\mathrm{R}}}
\newcommand\I{_{\mathrm{I}}}

\newcommand{\triple}[1]{{\left\vert\kern-0.25ex\left\vert\kern-0.25ex\left\vert #1
        \right\vert\kern-0.25ex\right\vert\kern-0.25ex\right\vert}}

\usepackage[linkcolor=blue,colorlinks=true, citecolor=blue]{hyperref}
\let\originaleqref\eqref 
\makeatletter
\renewcommand{\eqref}[1]{%
  \begingroup%
  \let\ref\@refstar%
  \hyperref[#1]{%
    Equation%
    ~\originaleqref{#1}%
  }%
  \endgroup
}
\makeatother

\usepackage[toc, page]{appendix}

\newcommand{\bR}{\mathbb{R}}
\newcommand{\bE}{\mathbb{E}}
\newcommand{\bP}{\mathbb{P}}
\newcommand{\bPa}{\mathbb{P}_a}
\newcommand{\bPf}{\mathbb{P}_f}

\newcommand{\tmu}{\widetilde{\mu}}

\makeatletter
\renewcommand*{\@fnsymbol}[1]{\ensuremath{\ifcase#1\or *\or \star\or ***\or
   \mathsection\or \mathparagraph\or \|\or **\or \dagger\dagger
   \or \ddagger\ddagger \else\@ctrerr\fi}}
\makeatother

\usepackage{fancyhdr}
\newtheorem{theorem}{Theorem}[section]
\newtheorem{lemma}[theorem]{Lemma}

\newtheorem{proposition}[theorem]{Proposition}

\theoremstyle{definition}
\newtheorem{remark}[theorem]{Remark}

\numberwithin{equation}{section}

\def\id{\mathrm{id}}

\renewcommand{\SS}{{\mathbb{S}^{n-1}}}

\newcommand\nablaS{\nabla_{\mathbb{S}}}
\newcommand\DeS{\De_{\mathbb{S}}}

\renewcommand\leq\leqslant
\renewcommand\geq\geqslant

\title[Non-identically distributed monochromatic random waves]{Asymptotics for the nodal components\\ of non-identically distributed\\ monochromatic random waves}

\author{Alberto Enciso}
\address{Instituto de Ciencias Matem\'aticas, Consejo Superior de
  Investigaciones Cient\'\i ficas, 28049 Madrid, Spain}
\email{aenciso@icmat.es}

\author{Daniel Peralta-Salas}
\address{Instituto de Ciencias Matem\'aticas, Consejo Superior de
 Investigaciones Cient\'\i ficas, 28049 Madrid, Spain}
\email{dperalta@icmat.es}

\author{\'Alvaro Romaniega}
\address{Instituto de Ciencias Matem\'aticas, Consejo Superior de
 Investigaciones Cient\'\i ficas, 28049 Madrid, Spain}
\email{alvaro.romaniega@icmat.es}

\begin{document}
\maketitle
\begin{abstract}
  We study monochromatic random waves on~$\RR^n$ defined by Gaussian
  variables whose variances tend to zero sufficiently fast. This has
  the effect that the Fourier transform of the monochromatic wave is
  an absolutely continuous measure on the sphere with a suitably
  smooth density, which connects the problem with the scattering
  regime of monochromatic waves. In this setting, we compute the
  asymptotic distribution of the nodal components of
  random monochromatic waves, showing that the number of nodal
  components contained in a large ball~$B_R$ grows asymptotically like
  $R/\pi$ with probability $p_n>0$, and is bounded uniformly in~$R$
  with probability $1-p_n$ (which is positive if and only if
  $n\geq3$). In the latter case, we show the existence of a unique
  noncompact nodal component. We also provide an explicit sufficient
  stability criterion to ascertain when a more general Gaussian
  probability distribution has the same asymptotic nodal distribution
  law.
\end{abstract}

\section{Introduction}

The Nazarov--Sodin theory, whose original motivation was to understand
the nodal set of random spherical harmonics of large
order~\cite{NS09}, has been significantly extended to derive
asymptotic laws for the distribution of the zero set of smooth
Gaussian functions of several variables. The primary examples are the
restriction to large balls of translation-invariant Gaussian functions
on~$\RR^n$ and various Gaussian ensembles of large-degree polynomials
on the sphere or on the torus.

From the point of view of applications, a particularly relevant
problem that falls within this framework is that of monochromatic
random waves, that is, solutions to the Helmholtz equation on~$\RR^n$ ($n\geq2$):
\begin{equation}\label{Helmholtz}
\De u + u=0\,.
\end{equation}
Since it is well known that any polynomially bounded solution to this
equation is the Fourier transform of a distribution supported on the
unit sphere~$\SS$, the way one constructs monochromatic random waves
is the following~\cite{CS19}. One starts with a real-valued orthonormal basis of
spherical harmonics on~$\SS$, which we denote by $Y_{lm}$. Hence
$Y_{lm}$ is an eigenfunction of the spherical Laplacian with
eigenvalue $l(l+n-2)$, the index $l$ is a nonnegative integer and $m$
ranges from~1 to the multiplicity
$d_l := \frac{2l+n-2}{l+n-2} \binom{l+n-2}{l}$ of the corresponding
eigenvalue.

To consider a monochromatic
random wave, one now takes
\begin{subequations}\label{random}
\begin{equation}\label{random1}
f(\xi):= \sum_{l=0}^\infty\sum_{m=1}^{d_l} i^l \, a_{\ell m}\,Y_{lm}(\xi)\,,
\end{equation}
where $a_{l m}$ are independent random variables, and
defines~$u$ as the Fourier transform of~$f\, dS$, where $dS$ is
the area measure of the unit sphere~$\SS$. This is tantamount to setting
\begin{equation}\label{random2}
u(x)=(2\pi)^{\frac n2} \sum_{l=0}^\infty\sum_{m=1}^{d_l} a_{lm} \,
Y_{lm}\bigg(\frac x{|x|}\bigg) \,\frac{J_{l+\frac n2-1}(|x|)}{|x|^{\frac n2-1}}\,,
\end{equation}
\end{subequations}
Note that~$u$ is real-valued if
the random variables~$a_{lm}$ are. In the Nazarov--Sodin theory, one
assumes that the random variables $a_{lm}$ are independent standard
Gaussians (i.e., of zero mean and unit variance).

Let us denote by $N_u(R)$ (resp., $N_u(R;[\Si])$) the number of connected components of the nodal
set $u^{-1}(0)$ that are contained in the ball centered at the origin
of radius~$R$ (resp., and diffeomorphic to~$\Si$). Here $\Si$ is any smooth, closed, orientable
hypersurface $\Si\subset\RR^n$. It is obvious from the definition that $N_u(R;[\Si])$ only depends on
the diffeomorphism class~$[\Si]$ of the hypersurface. The central known results
concerning the asymptotic distribution of the nodal components of
monochromatic random waves can then be
summarized as follows (see also~\cite{Gayet,Wigmanbis,CS19} for related results):

\begin{theorem}\label{T.1}
Suppose that the random variables~$a_{lm}$ in~\eqref{random} are
independent standard Gaussian variables. Then:
\begin{enumerate}
\item Nazarov--Sodin's estimate for the number of nodal
  components~\cite{NS16}: There is a constant $\nu>0$ such that
  \[
\bP\Bigg(\lim_{R\to\infty}\frac{N_u(R)}{R^n}=\nu\Bigg)=1\,.
  \]

  \item Sarnak--Wigman's positive probability bound for the number of
   nodal sets of fixed topology~\cite{SW19}: For each smooth, closed, orientable
hypersurface $\Si\subset\RR^n$ there exists a constant $\nu([\Si])>0$,
depending only on the diffeomorphism class of~$\Si$, such that
\[
\bP\Bigg(\lim_{R\to\infty}\frac{N_u(R;[\Si])}{R^n}=\nu([\Si])\Bigg)=1\,.
\]
\end{enumerate}
\end{theorem}

\begin{remark}
More visually, this theorem asserts that, if $a_{lm}$ are independent
standard Gaussians, the number of nodal components contained in a
large ball is almost surely proportional to the volume. This
volumetric growth rate holds even if one only considers nodal
components of a fixed (compact) topology.
\end{remark}

Our objective in this paper is to understand the asymptotic
distribution of the nodal set of~$u$ when the random
variables~$a_{lm}$, which we will no longer assume to be identically
distributed, have different distribution laws. One obvious motivation
to consider this problem is that the Helmholtz equation~(\ref{Helmholtz}) plays a central role in Physics,
particularly in quantum
mechanics and electromagnetic theory via scattering problems and in stationary solutions of the
3D Euler equation through Beltrami fields~\cite{CK92,Acta,RS}. In these contexts (which
are clearly different from the study of high energy eigenfunctions on a
compact manifold and from problems in percolation theory), one is
interested in solutions with the sharp decay at infinity, which is captured
by imposing that the Agmon--H\"ormander seminorm
\[
\triple u:= \limsup_{R\to\infty}\bigg(\frac1R\int_{B_R}|u|^2\, dx\bigg)^{\frac12}
\]
is finite. As we recall in Appendix~\ref{A.regularity}, the decay
properties of~$u$ are closely related to the regularity of the
function~$f$ above; indeed, it is a classical result of
Herglotz~\cite[Theorem 7.1.28]{Hor15} that $\triple u<\infty$ if and
only if $u$ is the Fourier transform of a measure of the form $f\,
dS$ with $\|f\|_{L^2(\SS)}<\infty$.

However, it is easy to see that, when $a_{lm}\sim \cN(0,1)$ are
standard Gaussians, $f$ is almost surely not in~$L^2(\SS)$ by the law
of large numbers. This means that this choice of random variables is
very well suited to the study of random eigenfunctions on a compact
manifold, as it is known since Nazarov and Sodin's breakthrough paper
on spherical harmonics~\cite{NS09}, but precisely for this reason, it
cannot capture the features of random solutions to non-compact
problems in the scattering regime (i.e., with finite
Agmon--H\"ormander seminorm). Hence one would like to consider, at
least, the case where
\[
a_{lm}\sim \cN(0,\si_l^2)
\]
are independent Gaussian variables of zero mean but distinct variances
$\si_l^2$. The fall-off (or growth) of the covariance~$\si_l$ as
$l\to\infty$ is directly related to the expected regularity of~$f$;
indeed, the easiest calculation in this direction is that the expected value of the $H^s(\SS)$~norm
of~$f$ is
\begin{equation}\label{variancecond}
\bE(\|f\|_{H^s(\SS)}^2)=\sum_{l=0}^\infty{d_l} (1+l)^{2s}\si_l^2\,.
\end{equation}
Since $d_l=c_n l^{n-2} + O(l^{n-3})$ for large~$l$, a convenient way
of stating our main result is as follows:

\begin{theorem}\label{T.2}
Suppose that the random variables $a_{lm}$ in~\eqref{random} are
independent Gaussians $\cN(0,\si_l^2)$, where the variances satisfy
\begin{equation}\label{convsil}
\sum_{l=0}^\infty  (1+l)^{2s+n-2}\si_{l}^2<\infty
\end{equation}
for some $s>\frac{n+5}2$. Then $f\in H^s(\SS)$ almost surely, so in
particular $\triple u<\infty$. Furthermore:
\begin{enumerate}

\item There exists some probability $p_n$, with $p_2=1$ and $p_n\in (0,1)$ if $n\geq3$,
  such that
  \begin{align*}
    \bP\bigg(\lim_{R\to\infty} \frac{N_u(R)}R =\frac1\pi\bigg)&= p_n\,,\\[1mm]
    \bP\bigg(\lim_{R\to\infty} {N_u(R)} <\infty \bigg)&= 1-p_n\,.
  \end{align*}

\item If $\Si\subset\RR^n$ is a smooth, compact, orientable
  hypersurface, then
  \begin{align*}
\bP\bigg(\lim_{R\to\infty}\frac{N_u(R;[\Si])}R=\frac1\pi\bigg)&=p_n  &\text{if } [\Si]=[\SS]\,,\\[1mm]
\bP\bigg(\lim_{R\to\infty}{N_u(R;[\Si])}<\infty \bigg)&=1-p_n   &\text{if } [\Si]=[\SS]\,,\\[1mm]
       \bP\bigg( \lim_{R\to\infty}{N_u(R;[\Si])}<\infty\bigg)&=1  &\text{if } [\Si]\neq[\SS]\,.
  \end{align*}

\end{enumerate}
\end{theorem}

\begin{remark}
In plain words, this theorem says that, when the variances satisfy
the convergence condition~(\ref{convsil}), the asymptotic
distribution is completely different from that of the Nazarov--Sodin
regime: the number of nodal components diffeomorphic to a sphere that
are contained in a large
ball grows like the radius with probability~$p_n$ and stays uniformly
bounded with probability~$1-p_n$. The number of non-spherical nodal
components stays uniformly bounded almost surely.
One can also study the nesting graph of the nodal structure, see~\cite{SW19,BMW19} for a definition. In the setting of Theorem~\ref{T.2}, with probability $p_n$, the nesting graph is a tree with degree 2 internal vertices, and the number of other trees is bounded almost surely.
\end{remark}

\begin{remark}\label{rem:isotropy}
Arguing as in Lemma~\ref{L.buy} using~\eqref{RandomU}, it is easy to show that the covariance kernel of our random field is
$$\bE\left(u(x)u(y) \right)=(2\pi)^n\sum_{\ell=0}^\infty \sigma_l^2\frac{|d_\ell|}{|\SS|}P_{\ell n}\left(\frac{x}{|x|}\cdot \frac{y}{|y|}\right)\frac{J_{\ell + \frac{n}{2}-1}(|x|)}{|x|^{\frac{n}{2}-1}}\frac{J_{\ell + \frac{n}{2}-1}(|y|)}{|y|^{\frac{n}{2}-1}}\,,$$
where $P_{\ell n}$ is the Legendre polynomial of degree $\ell$ in $n$
dimensions. Therefore, our random field is isotropic (i.e., invariant
under rotations) but not translation-invariant, in general. In the case when
$\sigma_{\ell}=1$ for all~$\ell$, one does have translational
invariance; indeed, a straightforward
computation~\cite{CS19} shows that the covariance kernel reduces to 
$$
\frac{J_{\frac{n}{2}-1}(|x-y|)}{|x-y|^{\frac{n}{2}-1}}
$$
up to a multiplicative constant.
\end{remark}

To offer some perspective into the ideas behind Theorem~\ref{T.2}, it
is convenient to start by recalling the gist of the proof of the first part of Theorem~\ref{T.1}. Nazarov
and Sodin start off with a clever (non-probabilistic) ``sandwich estimate'' of the form
\[
\bigg(1-\frac rR\bigg)^n \Bint_{B_{R-r}}\frac{N_{\tau_x
    u}(r)}{r^n}\, dx\leq  \frac{N_u(R)}{R^n}\leq \bigg(1+\frac rR\bigg)^n \Bint_{B_{R+r}}\frac{N_{\tau_x
    u}(r)+ \mathfrak N_{\tau_x
    u}(r)}{r^n}\, dx\,,
\]
where $\tau_xu(y):= u(x+y)$ is a translation of~$u$ and
$\mathfrak N_u(r)$ denotes the number of critical points of the
restriction $u|_{\pd B_r}$. Now one can exploit the fact that, in the
particular case when $a_{lm}$ are independent standard Gaussians, $u$
is a Gaussian random function with translation-invariant distribution,
which is the setting that the Nazarov--Sodin theory applies to.
Moreover, its spectral measure (which is simply~$dS$, the normalized
area measure on~$\SS$) has no atoms. Therefore, a theorem of Grenander,
Fomin and Maruyama and the Kac--Rice bound respectively imply that the
action of shifts on~$u$ is ergodic and that the expected value of~$\mathfrak N_u(r)$ is of order $r^{n-1}$. By taking limits
$1\ll r\ll R$, this readily implies the existence of
$\lim_{R\to\infty} N_u(R)/R^n$. The fact that this limit is positive
then follows from the sandwich estimate and the existence of a
(non-random) solution with a structurally stable compact nodal
set. Let us stress that the whole theory hinges on the fact that
$a_{lm}$ are Gaussians of the same variance, as this is crucially
employed both to connect the problem with the theory of Gaussian
random functions and to show that one can compute limits using ergodic
theory. The second item in Theorem~\ref{T.1} uses that, in fact, one can prescribe the
topology of a robust nodal component~\cite{EPS13}.

It should then come as no surprise that the proof of Theorem~\ref{T.2}
is based on entirely different principles. The basic idea is that,
with probability~1, in the setting of Theorem~\ref{T.2} the
density~$f$ is an $H^s(\SS)$-smooth function (and, as $s>\frac{n+5}{2}$, of class $C^3$ by the Sobolev embedding theorem) with nondegenerate zeros, and
that the probability~$p_n$ that~$f$ does not vanish is strictly
positive. When~$f\in H^s(\SS)$ does not vanish, it is not hard to prove using
asymptotic expansions that the number of nodal components contained in
a large ball~$B_R$ grows as the radius, and that all but a uniformly
bounded number of them are diffeomorphic to a sphere. When the zero
set of~$f$ is regular and nonempty, one can show that the number of nodal components
on~$\RR^n$ is bounded. However, the analysis is considerably
subtler because it hinges on the stability of certain noncompact
components of the nodal set that locally look like a helicoid. Putting these facts together, one heuristically arrives at
Theorem~\ref{T.2}.

It is worth mentioning that the regularity effect that we have striven to
capture is completely different from the use of frequency-dependent
weights considered by Rivera in the context of random Gaussian fields on compact
manifolds~\cite{Rivera19} (see also~\cite{FLL} for frequency-dependent
weights in the context of
random algebraic hypersurfaces). Indeed, Rivera's central result is that the
Nazarov--Sodin asymptotics still holds, with different constants, for
series of the form
\[
F(x):= \sum_{k=1}^L \lambda_k^{-s} \, a_k \, e_k(x)\,,
\]
where $L\gg1$, $(e_k,\lambda_k)$ are the eigenfunctions and eigenvalues of the
Laplacian on a compact $n$-manifold, $s\in (0,\frac n2)$ and
$a_l\sim \cN(0,1)$. In contrast, we
are interested in regimes with a different asymptotic behavior that
correspond to a scattering situation on~$\RR^n$.

It is natural to wonder which kind of asymptotic laws may arise from
more general randomizations of the function~$f$. As a first step in
this direction, we state next a ``stability result'',
that is, sufficient conditions for the asymptotics of
Theorems~\ref{T.1} and~\ref{T.2} to hold for more general probability
measures on the space of functions~$f$ (or~$u$). These conditions are
by no means obvious a priori, but the proof is based on an elementary
idea: if two probability measures~$\mu$ and~$\tmu$ (on the space of functions on
the sphere, which one can identify with a space of
sequences~$\RR^\mathbb{N}$) are equivalent (i.e., mutually absolutely
continuous), then these measures have the same zero-probability
events. The aforementioned sufficient conditions are then derived by
imposing that one of these measures correspond to the
Nazarov--Sodin distribution or to the distributions considered in
Theorem~\ref{T.2}.

\begin{theorem}\label{T.3}
Suppose that there is a nonnegative integer~$l_0$ and reals $M_{lm}$ and
$\si_{lm}$ such that the random variables~$a_{lm}$ in~\eqref{random}, which we assume to be
independent, follow any probability distribution on the line (absolutely
continuous with respect to the Lebesgue measure) for~$l<l_0$ and
Gaussian distributions $\cN(M_{lm},\si_{lm}^2)$ for $l\geq l_0$. Then:
\begin{enumerate}
\item The results of Theorem~\ref{T.1} hold, with the same constant~$\nu$,
  if
  \[
\sum_{\ell=l_0}^\infty \sum_{m=0}^{d_l}\Bigg[\frac{M_{lm} ^2}{
  \sigma_{\ell m} ^2+1 } + \frac{(\sigma_{lm}-1) ^2}{ \sigma_{lm} }\Bigg]<\infty\,.
    \]
 \item The results of Theorem~\ref{T.2} hold
  if there are constants $\si_l$ satisfying~(\ref{convsil}) such that
  \[
\sum_{\ell=l_0}^\infty \sum_{m=0}^{d_l}\bigg[\frac{M_{lm}^2}{
  \left(\sigma_\ell ^2+\sigma_{lm} ^2\right)} + \frac{(\sigma_{\ell} -\sigma_{lm}) ^2}{\sigma_{\ell} \sigma_{lm} }\bigg]<\infty\,.
    \]
\end{enumerate}
\end{theorem}

\section{The Fourier transform of measures on the sphere with an $H^s$-smooth
  density}
\label{S.FT}

Our goal in this section is to obtain sharp asymptotic expansions for
the Fourier transform
\[
u:=\widehat{f\, dS}
\]
of measures of the form $f\, dS$,
where for the time being we can think of the integrable function $f:\SS\to\CC$ simply as a series of spherical harmonics:
\[
f(\xi):= \sum_{l=0}^\infty\sum_{m=1}^{d_l} f_{\ell m}\,Y_{lm}(\xi)\,.
\]
It is well known that, for any real~$s$, the $H^s(\SS)$~norm
of~$f$ can then be computed as
\[
\|f\|_{H^s(\SS)}^2=\sum_{l=0}^\infty\sum_{m=1}^{d_l} (1+l)^{2s}|f_{\ell m}|^2\,.
\]
We want~$u$ to be real-valued, so we impose that
\[
f_{lm}:= i^l a_{lm}
\]
with $a_{lm}\in\RR$. The real and imaginary parts of $f$ are then
respectively given by the terms where $l$ is even and odd:
\begin{align*}
f\R(\xi) &:= \sum_{l\text{ even}}\sum_{m=1}^{d_l} (-1)^{\frac l2}a_{\ell
  m}\,Y_{lm}(\xi)\,,\\
  f\I(\xi) &:= \sum_{l\text{ odd}}\sum_{m=1}^{d_l} (-1)^{\frac {l-1}2}a_{\ell m}\,Y_{lm}(\xi)\,.
\end{align*}

To analyze $u$, we shall start by recalling the explicit formula for the Fourier
transform of a spherical harmonic, which we borrow
from~\cite{CS19}. For the benefit of the reader, we include a short
proof that only employs classical formulas for special functions,
instead of the theory of point pair invariants and zonal
spherical functions. For the ease of notation, here and in what
follows we set
\[
\La:= \frac n2-1\,.
\]
Also, throughout we will often denote the radial and angular parts of~$x$ by
\[
r:=|x|\in\RR^+\,,\qquad \te:=\frac x{|x|}\in\SS\,.
\]

\begin{proposition}\label{P.sh}
  The Fourier transform of the measure $Y_{lm}\, dS$ is
\begin{equation}\label{FouSH}
\widehat{Y_{\ell m}\, dS}(x)=(2\pi)^{\frac{n}{2}}\,(-i)^\ell\, Y_{\ell
  m} \left(\frac x{|x|}\right) \frac{J_{\ell + \La}(|x|)}{|x|^{\La}}\,,
\end{equation}
where $J_\alpha$ is the Bessel function of the first kind.
\end{proposition}
\begin{proof}
By the Funk–Hecke formula~\cite[Theorem 2.22]{AH12}, we have
\begin{equation}
\widehat{Y_{\ell m}\, dS}(x)=c_{\ell}(r)Y_{\ell m}(\te)\,,
\end{equation}
where
\begin{equation}
c_{\ell}(r)
= |\mathbb S^{n-2}|\int_{-1}^1 e^{-itr} \, P_{\ell n}(t) (1-t^2)^{\frac{n-3}{2}} \, dt\,.
\end{equation}
Here $P_{\ell n}$ is the Legendre polynomial. In turn, this last integral can be calculated using the
formula~\cite[Proposition 2.26]{AH12}:
\[
\int_{-1}^1 e^{-itr}\, P_{ln}(t)\, (1-t^2)^{\frac{n-3}2}\, dt =
\frac{(-ir)^l \, \Ga(\frac{n-1}2)}{2^l \Ga(l+\frac{n-1}2)} \int_{-1}^1
e^{-itr} (1-t^2)^{l+\La-\frac{1}2}\, dt\,.
\]
The proposition then follows in view of the well-known integral representation of the Bessel function,
\begin{equation}
J_{\al}\left(z\right)=\frac{(\tfrac{z}{2})^{\al}}{\pi^{\frac{1}{2}}%
\Gamma\left(\al+\tfrac{1}{2}\right)}\int_{-1}^{1} e^{-itz}\,(1-t^{2})^{\al-\frac{1}{2}}%
\, {d}t,
\end{equation}
\end{proof}

While the obtention of an asymptotic expansion for the Fourier
transform of the measure $f\, dS$ hinges on the analysis of
oscillatory integrals, it is convenient to employ the structure of the
problem to obtain sharper results. This will be done by exploiting the
expansion in spherical harmonics and then using asymptotics with
uniform constants directly for Bessel functions. It is worth pointing
out that, by blindly following the general approach to
asymptotic expansions (e.g., \cite[Theorem~7.7.14]{Hor15}), one would
need~$f \in H^s(\SS)$ with $s>\frac{3n+1}{2}$ (without considering derivatives), while the approach we take here
will lower this number to $s>\frac{n+5}2$.

Let us denote by
\[
\pd_r u:= \frac x{|x|}\cdot \nabla u\,,\qquad \snabla u:= \nabla u- \frac{x\cdot \nabla u}{|x|^2}x
\]
the radial and angular parts of the gradient. The covariant derivative
on the unit sphere will be denoted by~$\nablaS$.

\begin{proposition}\label{P.stationary}
If $f\in H^s(\SS)$ with $s>\frac{n+5}2$, then
\begin{align*}
  u&=
  \frac{2(2\pi)^{\frac{n-1}{2}}}{r^{\frac{n-1}{2}}}\big[f\R(\te)\,\cos(r-r_0)+
        f\I(\te)\, \sin(r-r_0)+ \cE_1(r)\big]\,,\\
  \pd_r u&=
  \frac{2(2\pi)^{\frac{n-1}{2}}}{r^{\frac{n-1}{2}}}\big[ -f\R(\te)\,\sin(r-r_0)+
              f\I(\te)\, \cos(r-r_0) +\cE_2(r)\big]\,,\\
  \snabla u &=
              \frac{2(2\pi)^{\frac{n-1}{2}}}{r^{\frac{n+1}{2}}}\big[\nablaS
              f\R(\te)\,\cos(r-r_0)+
        \nablaS f\I(\te)\, \sin(r-r_0)+ \cE_3(r)\big]\,,
\end{align*}
where $r_0:= \frac{\pi}{4} (n-1)$ and
the errors are bounded as
\[
  |\cE_1(x)|+ |\nabla\cE_1(x)|+  |\cE_2(x)|
  + |\cE_3(x)|\leq \frac{C\|f\|_{H^{s}(\SS)}}{r}\,.
\]
\end{proposition}

\begin{proof}
  By Proposition~\ref{P.sh}, $u$ is given by the series
  \begin{align*}
u(x) &= \sum_{l=0}^\infty\sum_{m=1}^{d_l} i ^l a_{lm} \widehat{Y_{lm}\,
    dS}(x)\\
    &= \frac{(2\pi)^{\frac n2}}{r^\La} \sum_{l=0}^\infty\sum_{m=1}^{d_l} a_{lm}
      Y_{lm}(\te)\, {J_{l+\La}(r)}\,.
  \end{align*}
  Let us now recall the following uniform bound for a Bessel
  function~\cite[Theorem 4]{Kra14}, valid for all $\al\geq0$ and $z\geq0$:
  \begin{equation}\label{jal}
J_\alpha(z) = \sqrt{\frac{2}{\pi z}} \, \cos\bigg(z-\frac{\pi \alpha}{2}-\frac{\pi}{4}\bigg) +  \bigg|\al^2-\frac14\bigg|\,\theta_\al(z)\, z^{-3/2}\,,
\end{equation}
where $|\theta_\al(z)|\leq 1$. Setting
  \begin{align*}
u_1(x)&:= \frac{2(2\pi)^{\frac{n-1}{2}}}{r^{\frac{n-1}{2}}}\sum_{l=0}^\infty\sum_{m=1}^{d_l} a_{lm}
      Y_{lm}(\te)\, \cos\bigg(r-r_0-\frac{\pi
        l}{2}\bigg)\\
    &\phantom{:}= \frac{2(2\pi)^{\frac{n-1}{2}}}{r^{\frac{n-1}{2}}}\big[f\R(\te)\,\cos(r-r_0)+
        f\I(\te)\, \sin(r-r_0)\big]\,,
  \end{align*}
  it then follows that
  \[
\cE_1(x):= \frac{r^{\frac{n-1}{2}}}{2(2\pi)^{\frac{n-1}{2}}}\,[u(x)-u_1(x)]
  \]
  can be estimated as  \begin{align*}
|\cE_1(x)| &\leq \frac  {C}{r} \sum_{l=0}^\infty\sum_{m=1}^{d_l}
             (l+1)^2 |a_{lm}| |Y_{lm}(\te)| \,.
  \end{align*}
  Using the Cauchy-Schwarz inequality,  we then infer
  \[
|\cE_1(x)| \leq \frac  C{r} \sum_{l=0}^\infty  (1+l)^2\bigg(\sum_{m=1}^{d_l}
            |Y_{lm}|^2\bigg)^{1/2} \bigg(\sum_{m=1}^{d_l}
            |a_{lm}|^2\bigg)^{1/2}\,.
  \]
  The addition theorem~\cite[Theorem 2.9]{AH12} ensures that, at any point on the sphere,
  \begin{equation}\label{addition}
\sum_{m=1}^{d_l}
            |Y_{lm}|^2= c_{ln}
          \end{equation}
          with an explicit constant bounded as $c_{nl}\leq C(l+1)^{n-2}$. This then allows us
to write
\[
|\cE_1(x)| \leq \frac  C{r} \sum_{l=0}^\infty  (1+l)^{\frac n2+1}\bigg(\sum_{m=1}^{d_l}|a_{lm}|^2\bigg)^{1/2}\,.
\]
Applying Cauchy--Schwarz again we obtain
\[
|\cE_1(x)| \leq \frac  C{r} \bigg(\sum_{l=0}^\infty (1+l)^{
  n-2s+2}\bigg)^{1/2}\bigg(
\sum_{l=0}^\infty\sum_{m=1}^{d_l}(1+l)^{2s}|a_{lm}|^2\bigg)^{1/2}\leq \frac Cr\|f\|_{H^{s}(\SS)}\,,
\]
as claimed.

Let us now compute the radial derivative of~$u$. We start by noting that
\[
\pd_r u = {(2\pi)^{\frac n2}} \sum_{l=0}^\infty\sum_{m=1}^{d_l} a_{lm}
      \pd_r\bigg[Y_{lm}(\te)\, \frac{J_{l+\La}(r)}{r^\La}\bigg]\,.
    \]
    Since
    \begin{equation}\label{rollo}
      \pd_r\bigg(\frac{J_{l+\La}(r)}{r^\La}\bigg)=
      \frac{J_{l+\La-1}(r)}{r^\La}-(l+2\La)\, \frac{J_{l+\La}(r)}{r^{\La+1}}
    \end{equation}
    and this formula depends solely on Bessel functions, one can now
    use again the uniform estimate~\eqref{jal} to derive, with the same
    reasoning as above, that the error
\[
\cE_2(x):= \frac{r^{\frac{n-1}{2}}}{2(2\pi)^{\frac{n-1}{2}}}\,[\pd_r u(x)-\pd_r u_1(x)]
\]
is bounded as
\[
  |\cE_2(x)|\leq \frac Cr\|f\|_{H^{s}(\SS)}\,.
\]

The bound for the angular part of the gradient can be estimated using
the same argument and the formula
\[
\snabla u = \frac{(2\pi)^{\frac n2}}{r^{\La+1}} \sum_{l=0}^\infty\sum_{m=1}^{d_l} a_{lm}
      \nablaS Y_{lm}(\te)\, {J_{l+\La}(r)} \,,
    \]
    the only difference being that instead of the addition
    formula~(\ref{addition}) one has to use that
    \[
\sum_{m=1}^{d_l} |\nablaS Y_{lm}|^2= l(l+n-2)\,c_{nl}\,.
    \]
    To prove this, it is enough to note that, by~\eqref{addition},
    \[
0=\DeS c_{ln}= \DeS \sum_{m=1}^{d_l}  Y_{lm}^2 = 2\sum_{m=1}^{d_l}
Y_{lm}\, \DeS Y_{lm}  + 2\sum_{m=1}^{d_l}
|\nablaS Y_{lm}|^2
\]
and use the eigenvalue equation $\DeS Y_{lm}= -l(l+n-2) Y_{lm}$. Using
now that
\[
  \nabla \cE_1=\cE_2\,\frac x r+\cE_3\,,
\]
the estimate for $\nabla\cE_1$ follows from the previous bounds.
    The proposition is then proved.
\end{proof}

\section{Nodal sets of non-random monochromatic waves}
\label{S.nonrandom}

We recall that the nodal set of a function~$F: M\to \RR^m$, where
$M$ is a manifold, is {\em regular}\/
if the derivative $(D F)_x: T_xM\to \RR^m$ has maximal rank
for all~$x\in F^{-1}(0)$. We say that a codimension one compact
submanifold $\cS$ of $\RR^n$ is a {\em graph over the sphere}\/ of radius $R$
centered at the origin if it can be written in spherical coordinates
$(r,\te)\in (0,\infty)\times\SS$ as
\[
\cS=\big\{(r,\te): r= R+ G(\te)\,,\; \te\in\SS\big\}
\]
for some smooth function $G:\SS\to (-R,\infty)$. In particular, $\cS$ is diffeomorphic to $\SS$.

\begin{theorem}\label{T.nodal}
Let $f\in H^s(\SS)$ with $s>\frac{n+5}2$ and denote by~$u$ the Fourier
transform of $f\, dS$. Then:
\begin{enumerate}
\item Suppose that $f$ does not vanish on~$\SS$. Then the
  nodal set $u^{-1}(0)$ has countably many connected components
  \[
u^{-1}(0)= \bigcup_{k=1}^\infty \cS_k\,,
  \]
 and for all large enough~$k$, $\cS_k$ is a graph over a sphere centered at the origin and is contained in the annulus $k\pi-c <|x|< k\pi+c$ for some
 constant~$c$ depending on~$f$. Furthermore,
 \[
\lim_{R\to\infty} \frac{\#\{k: \cS_k\subset B_R\}}R = \frac1\pi\,.
 \]

  \item Suppose that the zero set $f^{-1}(0)$ is a nonempty regular
    subset of the sphere. Then there is a large enough~$R$ such that
    $u^{-1}(0)\backslash B_R $ is connected.
\end{enumerate}
\end{theorem}

\begin{proof}
  By Proposition~\ref{P.stationary}, $u$~can
then be written as
\begin{equation}\label{uUE}
  u=
  \frac{2(2\pi)^{\frac{n-1}{2}}}{r^{\frac{n-1}{2}}}(U+\cE_1)\,,
\end{equation}
where
\[
  U:= f\R(\te)\, \cos(r-r_0) + f\I(\te)\, \sin(r-r_0)\,,
\]
and we have the bound
\begin{equation}\label{small}
  |\cE_1|+ |\nabla \cE_1|<C/r\,.
\end{equation}
It is clear that the zero sets of~$u$ and of $U+\cE_1$ coincide, so we
shall next study the latter.

  Let us begin with the first case.
 Since $f$ does not vanish, its
  modulus and phase functions, defined as
  \[
    f(\te)=:|f(\te)|\, e^{i\Te(\te)}\,,
  \]
  are of class $H^s(\SS)$, and $U$ can be equivalently written as
  \[
    U= |f(\te)|\,
    \cos[r-r_0-\Te(\te)]\,.
    \]
As $\min_{\te\in \SS} |f(\te)|>0$, the zero set
of~$U$ is given, in polar coordinates and for certain $k_0\in\ZZ$, by
\[
U^{-1}(0)=\bigcup_{k\geq k_0}\mathcal U_k\,,
\]
where
\[
\mathcal U_k:= \big\{ (r,\te)\in \RR^+\times\SS : r= \Te(\te)+
(k+\tfrac{n+1}4)\pi\big\}\,.
\]
Obviously
\[
\lim_{R\to\infty} \frac{\#\{k: \mathcal U_k\subset B_R\}}R=\frac1\pi\,.
\]

For large $k$, the component~$\mathcal U_k$ of the zero set is nondegenerate because
\[
\min_{x\in \mathcal U_k}|\pd_r U(x)|= \min_{\te\in \SS} |f(\te)|>0\,.
\]
In view of the bound~(\ref{small}),
Thom's isotopy theorem (see
e.g.~\cite[Theorem 3.1]{EPS13}) then
ensures that, outside a certain large compact set~$K$ containing the origin, the nodal set $u^{-1}(0)$ can be written as
\[
u^{-1}(0)\backslash K=\bigcup_{k\geq k_0}\mathcal \cS_k\,,
\]
where each connected component $\cS_k$ is of the form
\[
\cS_k= \Phi_k(\mathcal U_k)\,,
\]
where $\Phi_k$ is a smooth diffeomorphism of~$\RR^n$ with
$\|\Phi_k-\id\|_{C^0(\RR^n)}<C/k$. As the number of nodal
components of~$u$ contained in~$K$ is finite because the function~$u$
is analytic, the first statement follows.

Let us now pass to the second statement. We can use again the
decomposition~(\ref{uUE}) and study the zero set of~$U$ in this
case. Since
\[
U^{-1}(0)= \{ (r,\te): f\R (\te)\, \cos(r-r_0)= - f\I(\te)\, \sin(r-r_0)\}\,,
\]
one has, on $ U^{-1}(0)$,
\[
(\pd_rU)^2= [-f\R(\te)\, \sin(r-r_0) + f\I(\te)\, \cos(r-r_0)]^2=
f\I(\te)^2+ f\R(\te)^2\,,
\]
so $\nabla U|_{U^{-1}(0)}$ can vanish at most at the points
$(r,\te)\in U^{-1}(0)$ such that $f(\te)=0$.

To show that $\nabla U$ is nonzero also at those points, it is enough
to notice that
\[
\snabla U= \frac {\nablaS f\R (\te)\, \cos(r-r_0) + \nablaS f\I(\te)\, \sin(r-r_0)}r\,,
\]
so $\snabla U\neq0$ at any point $(r,\te)$ with $f(\te)=0$ because the
set $f^{-1}(0)$ is regular (so the vectors $\nablaS f\R (\te)$ and
$\nablaS f\I (\te)$ are linearly independent). Therefore, one
concludes that the zero set of~$U$ is regular, and in fact
\[
|\nabla U|_{U^{-1}(-\ep,\ep)}\geq c>0
\]
for some small $\ep>0$ because the function~$U$ is periodic on~$r$. One
can now use an analog of Thom's isotopy theorem for noncompact
sets~\cite[Theorem 3.1]{EPS13} with the
bound~(\ref{small}) to obtain that, for any large enough~$R$, there exists a diffeomorphism
$\Phi_R$ of~$\RR^n$, with
\[
  \|\Phi_R-\id\|_{C^0(\RR^n)}<C/R\,,
\]
such that
\[
u^{-1}(0)\backslash B_R= \Phi_R[U^{-1}(0)\backslash B_R]\,.
\]

Therefore, it only remains to analyze what $U^{-1}(0)$ looks like,
outside a large ball. We claim that $U^{-1}(0)\backslash
B_R$ is a connected set. Indeed, when the point $(r,\te)\in U^{-1}(0)$ is such
that $f(\te)\neq0$, it follows from the proof of the first assertion
of the statement that $U^{-1}(0)$ is locally a graph over a sphere
centered at the origin. When $f(\te)=0$,  $U^{-1}(0)$ is locally a sort of helicoid. To
see this, we can take advantage of the fact that $f^{-1}(0)$ is a
regular set to introduce local coordinates $(y_1,\dots, y_{n-1})$ in a
neighborhood of the point~$\te$ in~$\SS$ such that $y_1:=f\R$ and
$y_2:= f\I$. Hence, defining functions $\rho(y_1,y_2)$ and
$\phi(y_1,y_2)$ as
\[
y_1+i y_2=: \rho\, e^{i\phi}\,,
\]
we readily obtain that one can write
\[
U= \rho\cos (r-r_0-\phi)
\]
locally with respect to the coordinates~$(y_1,\dots, y_{n-1})$ and for
all large enough~$r$. In this conical sector, the zero set of~$U$ then consists of the codimension~2 conical
set $\rho=0$ and the
helicoidal hypersurface
\[
r= r_0+\phi+\frac\pi 2\,.
\]
Both kinds of local description of $U^{-1}(0)$ obviously
cover the whole zero set and show that it is connected.
For the benefit of the reader, we have included Figure~\ref{FigNod2}
with an illustration of what this nodal set looks like in three dimensions.
\begin{figure}[t]
	 	\includegraphics[width=0.4\textwidth]{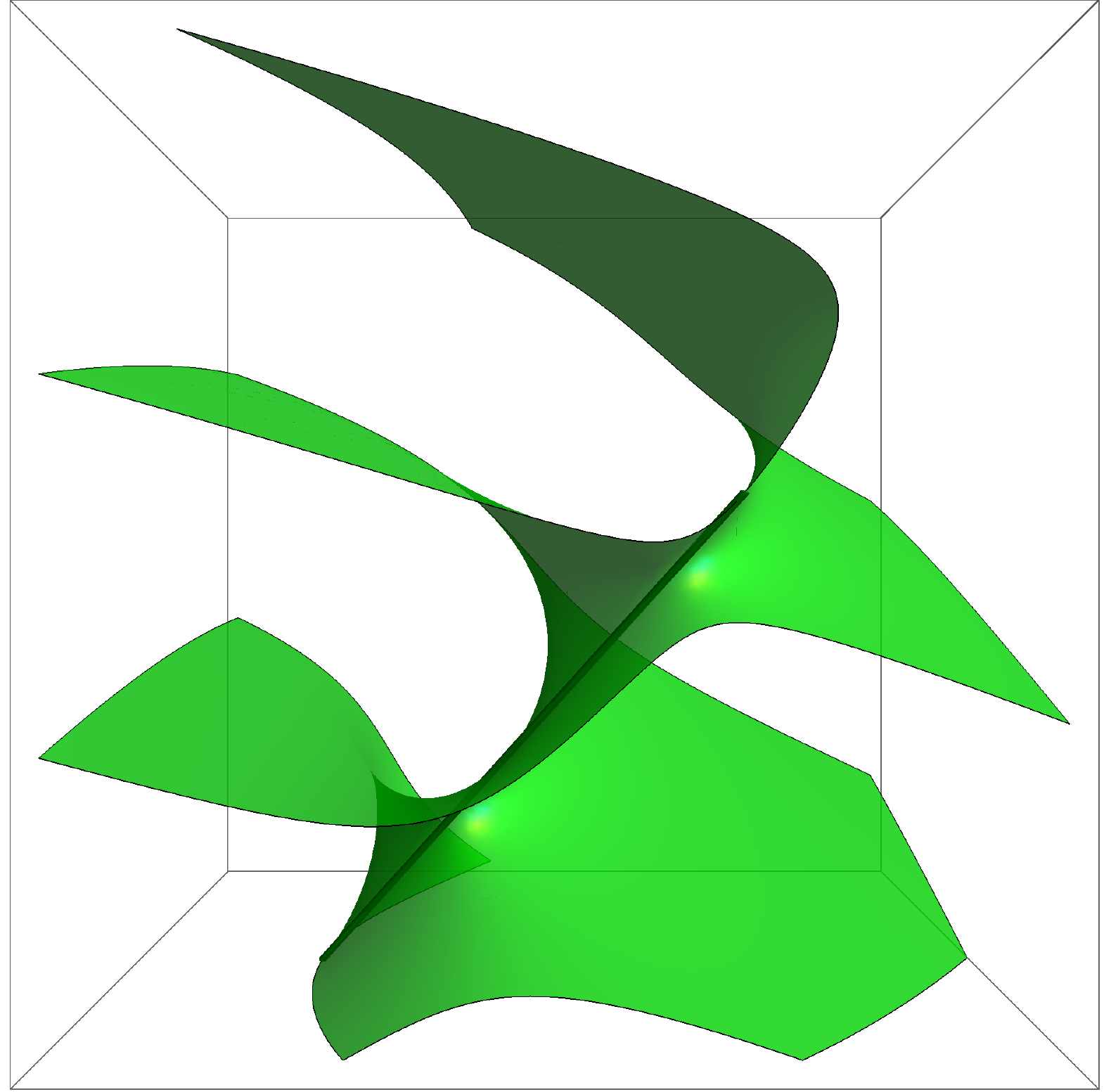}
	 \caption{Local structure of the zero set $u^{-1}(0)$ when $f$
         has regular zeros.} \label{FigNod2}
	\end{figure}
\end{proof}

\begin{remark}
If $s>\frac{n+5+2l}2$ for some integer $l\geq1$, one can conclude that
$\frac{r^{\frac{n-1}{2}}}{2(2\pi)^{\frac{n-1}{2}}} u$ is close to~$U$
in the~$C^{l+1}(\RR^n)$ norm, so~\cite[Theorem 3.1]{EPS13} then ensures that
$\|\Phi_k -\id\|_{C^l(\RR^n)}<C/k$. This immediately yields asymptotic
formulas for the area of each nodal component $\cS_k$.
\end{remark}

\section{Proof of Theorem~\ref{T.2}}
\label{S.Th2}

Let us start by introducing some notation associated with the
probabilistic setting described in~\eqref{random}. We denote by
$\bP_{lm}$ the probability
distribution of the random variable~$a_{lm}$, which we are assuming to
be a normal distribution of the form~$\cN(0,\si_l^2)$. By Kolmogorov's extension
theorem, the associated probability measure in~$\bR^\mathbb{N}$ is the product measure
that we will denote by $\bPa:=\prod_{\ell=0}^\infty \prod_{m=1}^{d_l}
\bP_{lm}$. The associated measures on the space of distributions on~$\SS$
and on~$\RR^n$ are respectively given by the pushed forward measures
$\bP_f:= f_*\bP_a$ and $\bP:= u_*\bP_a$, which we view as maps
\[
f:\omega\in \Omega\backslash \Omega_0\mapsto \sum_{\ell=0}^\infty
\sum_{m=1}^{d_l} i^l a_{\ell m}(\omega)\,Y_{lm}(\cdot)\in \mathcal D'(\SS)\,,
\]
and
\begin{equation}\label{RandomU}
u: \om\in\Omega\backslash\Omega_0 \mapsto (2\pi)^{\frac n2} \sum_{l=0}^\infty\sum_{m=1}^{d_l} a_{lm}(\om) \,
Y_{lm}\bigg(\frac \cdot{|\cdot|}\bigg) \,\frac{J_{l+\frac
    n2-1}(|\cdot|)}{|\cdot|^{\frac n2-1}}\in \mathcal D'(\RR^n)\,,
\end{equation}
where $\Omega_0\subset\Omega$ is a set of measure zero.

An important first observation is that, with the probability
distribution we are considering, $f$ is an $H^s$-smooth function with
probability~1:

\begin{lemma} \label{LemConVar}
  The function $f$ is of class~$H^s(\SS)$ almost surely.
\end{lemma}
\begin{proof}
The hypothesis~(\ref{variancecond}) implies that, for any~$L>0$, the expected value
of the finite sum is bounded by a uniform constant:
$$\bE\left(\sum_{\ell=0}^L \sum_{m=1}^{d_l} a_{\ell m}^2(\ell+1)^{2s}\right)=\sum_{\ell=0}^L {d_l} \sigma_\ell^2(\ell+1)^{2s}<C\,.$$
The monotone convergence theorem then ensures that
\[
\bE\big(\|f\|_{H^s(\SS)}^2\big)= \bE\bigg(\sum_{\ell=0}^\infty
\sum_{m=1}^{d_l} a_{\ell m}^2(\ell+1)^{2s}\bigg)<C\,,
\]
which implies that $\bP_f(H^s(\SS))=1$.
\end{proof}

The next result we need is that, again with probability~1, $f^{-1}(0)$
is a regular level set:

\begin{lemma}\label{L.buy}
The zero set of~$f$ is regular almost surely. Furthermore, if $n=2$,
almost surely~$f$ does not vanish.
\end{lemma}

\begin{proof}
Let us consider the vector field on the sphere $h(\theta,\lambda):= \nablaS f\R(\theta)-\lambda \nablaS f\I(\theta)$ for $\lambda\in\bR$. If we take local coordinates $(y_1,\dots, y_{n-1})$ around a
point~$\te\in\SS$, the components of~$h$ in this local chart are given by
\begin{equation}\label{localh}
h_{j}(\te,\lambda):=  \sqrt{g^{jj}}\pd_{y_j} f\R(\te)-\lambda \sqrt{g^{jj}}\pd_{y_j} f\I(\te)
\end{equation}
with $1\leq j\leq n-1$, as
\begin{equation}\label{eq:grad fa}
	\nablaS f_a=\sum_{i}{\partial_{y_i} f_a}\sqrt{g^{ii}}\mathbf {e} _{i}
\end{equation}
where $\mathbf {e} _{i}$ is the unit vector in our coordinates. Recall that, by Lemma~\ref{LemConVar}, $f\in C^3(\SS)$ almost surely.

In order to show that $(f\R, f\I, h)$ is a non-degenerate Gaussian vector field, we first analyze the probabilistic structure of $f$ and its derivatives, for which we need to compute the covariance matrix of ($f\R,f\I,\nablaS f\R,\nablaS f\I$). We recall that a non-degenerate Gaussian vector field means that the determinant of its covariance matrix is positive definite everywhere.

First, the covariance between $f\R$ (or derivatives of $f\R$) and
$f\I$ (or derivatives of $f\I$) is zero because they depend on
different independent coefficients, even and odd $l$
respectively. Second, since the Gaussian coefficients have zero mean,
the expected values of $f_a$ and $\nablaS  f_a$ are zero, where $f_a$
denotes either $f\R$ or~$f\I$. For the covariance kernel of $f_a$, notice that if $\te,\te'\in \SS$, we have
	\begin{equation}\label{Kernel}
	\bE\left(f_a(\te)f_a(\te') \right)=\sum_{\ell=0,\text{parity}=a}^\infty  \sigma_\ell^2 c_{ln}P_{\ell n}(\te\cdot \te')\,,
	\end{equation}
	where $P_{\ell n}$ is the Legendre polynomial of degree $\ell$
        in $n$ dimensions,  $c_{ln}$ was defined in~\eqref{addition}
        and the notation $\text{parity}=a$ means that the sum is
        restricted to even~$l$ if $a=\text{R}$ and to odd~$l$ if $a=\text{I}$. From the kernel~\eqref{Kernel} we can deduce the variance of $f_a$,
	$$\bE\left(f_a(\te)f_a(\te) \right)=\sum_{\ell=0,\text{parity}=a}^\infty  \sigma_\ell^2 c_{ln}\in(0,+\infty)\,,$$
	which is independent of $\te$ and finite by our hypothesis on $\sigma_\ell$.

Let us now prove that $f_a$ and its derivatives are independent. Indeed, the covariance between the function and a derivative reads as
	$$\bE\left(f_a(\te)\partial_{y_i}f_a(\te) \right)=\left.\sum_{\ell=0,\text{parity}=a}^\infty  \sigma_\ell^2 c_{ln}  P'_{\ell n}(\te\cdot \te')\te\cdot\partial_{y'_i}\te'\right|_{\te'=\te}=0\,,$$
	where we have used that $\te$ is a point on the unit sphere
        and hence $\te\cdot \te=1$. We also claim that the derivatives
        are independent. To prove it, we assume that
        the vector fields $\{\partial_{y_i}\}$ are orthogonal, i.e.,
        $g_{ij}=0$ if~$1\leq i< j\leq n-1$, where $g_{ij}$ is the
        induced metric on~$\SS$. This can be accomplished, for instance, by taking hyperspherical coordinates. If $g^{ij}$ denotes the inverse matrix of $g_{ij}$, we have
	\begin{align}
	&\bE\left(\sqrt{g^{ii}}\partial_{y_i}f_a(\te)\sqrt{g^{jj}}\partial_{y_j}f_a(\te) \right)= \nonumber\\
	&=\sqrt{g^{ii}}\sqrt{g^{jj}}\left.\sum_{\ell=0,\text{parity}=a}^\infty  \sigma_\ell^2 c_{ln}  \left[ P'_{\ell n}(\te\cdot \te')\partial_{y_i}\te\cdot\partial_{y'_j}\te'+P''_{\ell n}(\te\cdot \te')\left(\partial_{y_i}\te\cdot \te'\right)\left(\hspace{0.1mm} \te\cdot\partial_{y'_j}\te'\right)  \right]\right|_{\te=\te'}\nonumber\\
	&=\sum_{\ell=0,\text{parity}=a}^\infty  \sigma_\ell^2 c_{ln} P'_{\ell n}(1)(\partial_{y_i}\te)^2\delta_{ij}\frac{1}{g_{ii}}=\delta_{ij}\sum_{\ell=0,\text{parity}=a}^\infty \sigma_\ell^2 c_{ln} P'_{\ell n}(1)\label{eq:bE der}
	\end{align}
	where in the second and third equalities we have used the orthogonality condition of the coordinate system and the definition of the metric
	$$g_{ij}=\partial_{y_i}\te\cdot\partial_{y_j}\te\,.$$
	As before, this covariance matrix is strictly positive definite, independent of the point and finite. The finiteness follows from the differential equation satisfied by $P_{\ell n}$,
	$$\displaystyle \left(1-x^{2}\right)^{\frac{3-n}{2}}{\frac {d}{dx}}\left[\left(1-x^{2}\right)^{\frac{n-1}{2}}{\frac {dP_{\ell n}(x)}{dx}}\right]+\ell(\ell+n-2)P_{\ell n}(x)=0\,, $$
	which allows us to compute $P'_{\ell n}(1)=\frac{\ell(\ell+n-2)}{n-1}.$ Using \eqref{eq:grad fa} we conclude that the covariance matrix of ($f\R,f\I,\nablaS f\R,\nablaS f\I$) is diagonal and positive definite.

We are now ready to prove that $(f\R, f\I, h)$ is a non-degenerate
Gaussian vector field. To see this, first notice that the computations above imply that
$f\R$, $f\I$ and their derivatives are independent as they are uncorrelated (i.e., their covariance matrix vanishes), which ensures that
the Gaussian vector field~$(f\R, f\I, h)$ has zero mean (as linear combinations of independent Gaussian random variables are still Gaussian, our field is Gaussian). Also, the
local expression~(\ref{localh}) also ensures that
 $$\bE\left(f\R(x)h(x)\right)= \bE\left(f\I(x)h(x)\right) = 0\,.$$
By \eqref{eq:bE der},
\begin{align*}
	\bE\left(h_i(\theta,\lambda)h_j(\theta,\lambda)\right)&=\delta_{ij}\left(\sum_{\ell=0,\text{parity}=\text{ even}}^\infty \sigma_\ell^2 c_{ln} P'_{\ell n}(1)+\lambda^2 \sum_{\ell=0,\text{parity}=\text{ odd}}^\infty \sigma_\ell^2 c_{ln} P'_{\ell n}(1)\right)\\
	&=:\delta_{ij}(\sigma\R^2+\lambda^2\sigma\I^2).
\end{align*}
We can now use suitable generalizations of Bulinskaya's
lemma~\cite[Proposition 6.11]{AW09} to conclude that
\[
\bPf(\{\exists\, \te,\lambda: f(\te)=0,\; h(\te,\lambda)=0\}) =0\,.
\]
Indeed, $(f\R, f\I, h)$ is a non-degenerate Gaussian vector field going from an $n$-dimensional space to $\bR^{n+1}$ and it is $C^2(\SS)$ almost surely. As the covariance matrix determinant, $\det\Sigma(\lambda)$, attains its minimum value at $\lambda=0$, independent of $\theta$ and strictly positive, the density of $(f\R, f\I, h)$ at zero is bounded for all values of $\theta,\lambda$, i.e.,
$$\rho(x,\lambda)= {\displaystyle{\frac {\exp \left(-{\frac {1}{2}}x^{\mathrm {T} }{ {\Sigma(\lambda) }}^{-1}x\right)}{\sqrt {(2\pi )^{n+1} |{\det {\Sigma(\lambda) }}|}}}}\le \rho(0,0)<\infty.$$
This shows that the zero set of~$f$ is regular almost surely as $\nablaS f\R$ and $\nablaS f\I$ are linearly independent at $f^{-1}(0)$\footnote{It remains to consider the case $\nablaS f\I=0$ but $\nablaS f\R\neq 0$ at some point of $f^{-1}(0)$, but we can discard this event by the same reasoning applied to the easier case $(f\R,f\I,\nablaS f\I)$. }.

When $n=2$, the same argument applied to the Gaussian vector field
$(f\R,f\I)$ shows that
\[
\bPf(\{\exists\, \te: f(\te)=0\})=0\,,
\]
so with probability~1 the function~$f$ does not vanish, and the
lemma follows.
\end{proof}

In the next lemma we compute the probability that~$f$ does not vanish and
that $f$ has a nonempty regular zero set:

\begin{lemma}\label{PropPosProb} The probability that the
  function~$f$ does not vanish on~$\SS$ is~$p_2:=1$ if $n=2$ and $p_n\in
  (0,1)$ if $n\geq3$. Moreover, with probability $1-p_n$ the zero set $f^{-1}(0)$ is
  regular and nonempty.
\end{lemma}
\begin{proof}
  Given any function $f_0\in H^s(\SS)$, whose coefficients for the
  expansion in spherical harmonics we will denote by~$a_{lm}^0$, and any $\ep>0$, we claim that
  \begin{equation}\label{bPf}
\bPf\big(\{\|f-f_0\|_{H^s(\SS)}<\ep\}\big)>0\,.
\end{equation}
To prove this, we start by noting that we can take some~$L$, depending
  on~$\ep$, such that
$$\bP_a\left(\sum_{\ell=L}^\infty \sum_{m=1}^{d_l}|a_{lm}-a^0_{\ell
    m}|^2(\ell+1)^{2s}<\frac \ep2\right)>0\,, $$
which is obvious because~$f_0 \in H^s(\SS)$ and~$f$ is in~$H^s(\SS)$ almost surely by
Lemma~\ref{LemConVar}. \eqref{bPf} then follows because
\begin{multline*}
\bPf\big(\{\|f-f_0\|_{H^s(\SS)}<\ep\}\big)\geq\\ \bP_a\left(\sum_{\ell=L}^\infty \sum_{m=1}^{d_l}|a_{lm}-a^0_{\ell
    m}|^2(\ell+1)^{2s}<\frac\ep2\right)\, \bP_a\left(\sum_{\ell=1}^L \sum_{m=1}^{d_l}|a_{lm}-a^0_{\ell m}|^2(\ell+1)^{2s}<\frac\ep2\right)>0\,.
\end{multline*}

For all $n\geq2$, it then suffices to take $f_0:=1$ to conclude that
\[
p_n:=\bP_f(\{f>0\})>0\,;
\]
indeed, by Lemma~\ref{L.buy} one knows that $p_2=1$.
Likewise, when $n\geq3$, one can take any smooth function~$f_0$ whose
zero set is regular and nonempty to conclude, by the implicit function
theorem, that
\[
1-p_n=\bP_f(\{\min_{\SS}|f|=0\})>0\,.
\]
Notice that this argument does not work when $n=2$ because, as $f$ is
complex-valued, the rank of $\nabla f_0$ on~$f_0^{-1}(0)$ must be~2 to
apply the implicit function theorem. Finally, by Lemma~\ref{L.buy}, the nodal set is regular
almost surely, so the lemma follows.
\end{proof}

We are now ready to complete the proof of
Theorem~\ref{T.2}. Lemma~\ref{LemConVar} ensures that $f\in H^s(\SS)$
almost surely. Furthermore, by Lemma~\ref{PropPosProb}, with
probability~$p_n$, $f$~does not vanish, so in this case
Theorem~\ref{T.nodal} ensures that the nodal set of~$u$ has $R/\pi
+o(R)$ components diffeomorphic to~$\SS$ contained in~$B_R$ and only
$O(1)$ components that are not diffeomorphic to~$\SS$. Also by
Lemma~\ref{PropPosProb}, with probability $1-p_n$ the zero set
$f^{-1}(0)$ is regular and nonempty, so Theorem~\ref{T.nodal} ensures
that $N_u(R)=O(1)$. The theorem is then proved.

\section{Proof of Theorem~\ref{T.3}}
\label{S.stability}

Let us denote by~$\mu$ the probability measure on~$\RR^\mathbb{N}$ defined
by the random variables $a_{lm}$, which we now assume to be absolutely
continuous with respect to the Lebesgue measure for $l<l_0$ and
Gaussian distributions $\cN(M_{lm},\si_{lm}^2)$ for $l\geq l_0$. We
denote by $\bPa^0$ and~$\bPa$ the probability measures defined by
random variables $a_{lm}\sim
\cN(0,1)$ and $a_{lm}\sim \cN(0,\si_l^2)$ as in Theorem~\ref{T.2}, respectively.

To prove the theorem it is enough to show that in the first
(respectively, second) case, the measures
$\mu$ and $\bPa^0$ (respectively, $\bPa$) are mutually absolutely
continuous. Kakutani's dichotomy theorem, Proposition 2.21 in~\cite{dPZ14}, ensures that, in the first
case, these measures are mutually absolutely continuous
if and only if the Radon--Nikodym derivative of the measures satisfies
\begin{equation}\label{RN}
\prod_{\ell=0}^\infty \prod_{m=1}^{d_l} \int_{-\infty}^\infty
\bigg(\frac{d\mu_{lm}}{d\bPa^0}\bigg)^{1/2}\, d\bPa^0>0\,
\end{equation}
(being always $\le 1$). Since, for $l\geq l_0$,
\[
\frac{d\mu_{lm}}{d\bPa^0}(x)= \frac1{\si_{lm}} e^{\frac{x^2}2-\frac{(x-M_{lm})^2}{2\si_{lm}^2}}\,,
\]
one has
\[
\int_{-\infty}^\infty
\bigg(\frac{d\mu_{lm}}{d\bPa^0}\bigg)^{1/2}\, d\bPa^0=
(2\pi\si_{lm})^{-1/2}\int_{-\infty}^\infty
e^{-\frac{(x-M_{lm})^2}{4\si_{lm}^2}-\frac{x^2}4}\,dx= \bigg(\frac{2\si_{lm}}{1+\si_{lm}^2}\bigg)^{1/2}e^{-\frac{M_{lm}^2}{4+4\si_{lm}^2}}\,.
\]
Minus the logarithm of the product~(\ref{RN}) for $l\geq l_0$ is then given by the series
\[
\mathcal C:= \sum_{\ell=l_0}^\infty \sum_{m=1}^{d_l}
\frac{M_{lm}^2}{4+4\si_{lm}^2}+  \frac12\sum_{\ell=l_0}^\infty
\sum_{m=1}^{d_l} \log\frac{1+\si_{lm}^2}{2\si_{lm}}\,.
\]
As both terms are necessarily positive and using that a sequence
$a_n\geq1$ satisfies
\[
  \sum_{n=1}^\infty \log a_n<\infty \quad \text{if and only if}\quad
  \sum_{n=1}^\infty (a_n-1)<\infty\,,
\]
we then infer that that necessary
and sufficient condition for $\mathcal C<\infty$ (or, equivalently,
for the product~(\ref{RN}) to be nonzero) is that
\[
\sum_{\ell=l_0}^\infty \sum_{m=0}^{d_l}\Bigg[\frac{M_{lm} ^2}{
  \sigma_{\ell m} ^2+1 } + \frac{(\sigma_{\ell m}
    -1)^2}{\sigma_{\ell m} }\Bigg]<\infty\,.
\]
Likewise, $\mu$ and $\bPa$ are mutually absolutely continuous
if~\eqref{RN} holds with $\bPa^0$ replaced by~$\bPa$, which amounts to
\[
\sum_{\ell=l_0}^\infty \sum_{m=0}^{d_l}\bigg[\frac{M_{lm}^2}{
  \sigma_\ell ^2+\sigma_{lm} ^2} + \frac{(\sigma_{\ell} -\sigma_{lm}) ^2}{\sigma_{\ell} \sigma_{lm} }\bigg]<\infty\,.
    \]
Theorem~\ref{T.3} then follows.

\begin{remark}
 When the probability measure~$\mu$ is a general Gaussian measure
(not necessarily a product), the
 Feldman--Hajek theorem~\cite[Theorem 2.25]{dPZ14} characterizes when
 $\mu$ and~$\bPa$ (or~$\bPa^0$) are mutually absolutely
 continuous in terms of the
 mean and covariance operator of~$\mu$. However, the resulting
 condition is not very illustrative and we have opted not to include it. Nevertheless, this means that the results can be extended to coefficients which are not necessarily independent. Also, similar considerations using Kakutani's theorem can be applied to a product measure whose coefficients are not normal variables.
\end{remark}

\section*{Acknowledgments}

A.E.\ is supported by the ERC Starting Grant~633152. D.P.-S.
is supported by the grants MTM-2016-76702-P (MINECO/FEDER) and Europa Excelencia EUR2019-103821 (MCIU). A.R.\ is supported by the grant MTM-2016-76702-P (MINECO/FEDER). This work is supported in part by the ICMAT--Severo Ochoa grant
SEV-2015-0554 and the CSIC grant 20205CEX001.

\appendix
\section{The decay of~$u$ in terms of the regularity of~$f$}\label{A.regularity}

Standard arguments from the theory of distributions ensure that any
polynomially bounded solution to the Helmholtz equation
\[
\De u + u=0
\]
on~$\RR^n$ can be written as the Fourier transform of a distribution
supported on the unit sphere.
The fundamental result that connects the decay of the solution~$u$
with the regularity of its Fourier transform is a classical result
of Herglotz~\cite[Theorem 7.1.28]{Hor15}. In order to state it, let us denote by
\[
\triple u^2:= \limsup_{R\to\infty} \frac 1R\int_{B_R}u(x)^2\, dx
\]
the Agmon--H\"ormander seminorm of a function~$u$ on~$\RR^n$.

\begin{theorem}[Herglotz]
  A solution to the Helmholtz equation satisfies the decay condition
  \[
\triple u<\infty
  \]
  if and only if there is a function $f\in L^2(\SS)$ such that
  \begin{equation}\label{uf}
u= \widehat{f\, dS}\,.
\end{equation}
Furthermore, this decay estimate is sharp in the sense that there is a universal constant such that
  \[
\frac1C \triple u\leq \|f\|_{L^2(\SS)}\leq C\triple u\,.
  \]
\end{theorem}

An immediate consequence of this result is that the derivatives of any function of
the form~\eqref{uf} with $f\in L^2(\SS)$ have the same decay at
infinity. Indeed, for any~k,
\begin{equation}\label{condk}
\triple{\nabla^k u}\leq C\|\xi^k f\|_{L^2(\SS)}\leq C\| f\|_{L^2(\SS)}\,,
\end{equation}
and in general this is obviously sharp because $\De u = -u$.

However, it is not hard to see that higher regularity of~$f$ translates into higher decay rates
of the {\em angular}\/ derivatives of~$u$. In order to state this
result, let us denote by
\[
\snabla u:= \nabla u- \frac{x\cdot \nabla u}{|x|^2}x
\]
the angular part of the gradient and set
$\langle x\rangle := (1+|x|^2)^{1/2}$.

\begin{proposition}\label{P.regularity}
A solution to the Helmholtz equation satisfies the decay condition
  \[
\triple u_k:= \sum_{j=0}^k\triple{\langle x\rangle^j\, \snabla^j u}<\infty
  \]
  if and only if there is a function $f\in H^k(\SS)$ such that
  \begin{equation*}
u= \widehat{f\, dS}\,.
\end{equation*}
Furthermore, this decay estimate is sharp in the sense that there is a universal constant such that
  \[
\frac1C \triple u_k\leq \|f\|_{H^k(\SS)}\leq C\triple u_k\,.
  \]
\end{proposition}

\begin{proof}
  A simple integration by parts yields
  \[
\snabla^j u=\frac{1}{|x|^j}\widehat{\nablaS^j f\, dS}\,,
\]
so the result follows from Herglotz's theorem.
\end{proof}

\begin{remark}
Roughly speaking, Herglotz's theorem asserts that a solution~$u$ to the
Helmholtz equation on~$\RR^n$ can decay at most as
$|x|^{-\frac{n-1}2}$, on average, and that this sharp decay rate is
attained if and only if~$f$ is in~$L^2(\SS)$. Furthermore, this
proposition says that the $k^{\text{th}}$ angular derivatives of~$u$
can decay at the faster rate $|x|^{-\frac{n-1}2-k}$, and that
this sharp rate is attained in an $L^2$-averaged sense if and only if $f\in H^k(\SS)$.
\end{remark}

The case when $f$ is of lower regularity than~$L^2$, for instance
$f\in H^{-k}(\SS)$ for a positive integer~$k$, can be partly
understood with a similar reasoning. In this case, one can write
\[
f= \sum_{j=0}^k \mathcal L_j f_j
\]
with $f_j\in L^2(\SS)$ and $\mathcal L_j$ a differential operator on~$\SS$
of order~$j$ with smooth coefficients. Furthermore,
\[
\|f\|_{H^{-k}(\SS)}= \sum_{j=0}^k \| f_j\|_{L^2(\SS)}\,.
\]
Therefore, integrating by parts in the distributional formula
\[
u(x)=\int_{\SS} e^{ix\cdot \xi}\, f(\xi)\, dS(\xi)\,,
\]
one easily obtains that
\[
\frac1R\int_{B_R}\frac{u(x)^2}{1+|x|^{2k}}\, dx\leq C\|f\|_{H^{-k}(\SS)}^2\,.
\]
However one should note that, contrary to what happens in the previous results of this Appendix,
these are not the only solutions to the Helmholtz equation with this
decay rate. This is evidenced, e.g., by the solutions whose Fourier transform is
\[
\hu(\xi)= \delta^{(k)}(|\xi|-1)\,.
\]

\bibliographystyle{alpha}

\end{document}